\documentclass[11pt]{amsart}
\usepackage{a4wide,color}

\usepackage{amsfonts, amsmath, amscd}
\usepackage[psamsfonts]{amssymb}
\usepackage{amssymb,textcomp,lmodern}
\usepackage{pb-diagram}
\usepackage[all,cmtip]{xy}

\newtheorem{theorem}{Theorem}[section]

\newtheorem{corollary}[theorem]{Corollary}

\newtheorem{example}[theorem]{Example}

\newtheorem{problem}[theorem]{Problem}

\theoremstyle{definition}
\newtheorem{definition}[theorem]{Definition}
\newtheorem{remark}[theorem]{Remark}

\numberwithin{equation}{section}

\newcommand{\e}{\varepsilon}
\newcommand{\w}{\omega}

\newcommand{\IR}{\mathbb{R}}

\newcommand{\II}{\mathbb{I}}

\newcommand{\IF}{\mathbb{F}}

\newcommand{\Sz}{\mathrm{Sz}}

\newcommand{\KK}{\mathcal{K}}

\newcommand{\supp}{\mathrm{supp}}

\newcommand{\cl}{\mathrm{cl}}

\newcommand{\diam}{\mathsf{diam}}

\newcommand{\BNP}{\mathsf{BNP}}
\newcommand{\card}{\mathsf{card}}

\newcommand{\SM}{{\setminus}}

\input xy
\xyoption{all}

\title[Banach spaces with the Gelfand--Phillips property]{
Banach spaces with the (strong) Gelfand--Phillips property}
\author{Taras Banakh and Saak Gabriyelyan}

\address{Ivan Franko National University of Lviv (Ukraine) and Jan Kochanowski University in Kielce (Poland)}
\email{t.o.banakh@gmail.com}

\address{Department of Mathematics, Ben-Gurion University of the Negev, Beer-Sheva, P.O. 653, Israel}
\email{saak@math.bgu.ac.il}



\begin{document}

\begin{abstract}
Several new characterizations of the Gelfand--Phillips property are given. We define a strong version of the Gelfand--Phillips property and prove that a Banach space has this stronger property iff it embeds into $c_0$. For an infinite compact space $K$,  the Banach space $C(K)$ has the strong Gelfand--Phillips property iff $C(K)$ is isomorphic to $c_0$ iff $K$ is countable and has finite scattered height.
\end{abstract}

\maketitle


\section{Introduction}


All topological spaces are assumed to be Tychonoff and infinite, all Banach spaces are infinite-dimensional over the field $\mathbb{F}$ of real or complex numbers, and all operators between Banach spaces are linear and continuous. For a Banach space $E$, we denote by $B_E$ the closed unit ball of $E$, and the dual space of $E$ is denoted by $E'$.  For a bounded subset $B\subseteq E$ and a functional $\chi\in E'$,  we put $\|\chi\|_B:= \sup\{ |\chi(x)|:x\in B\}$.

A bounded subset $B$ of a Banach space $E$ is called {\em limited\/} if every weak$^\ast$ null sequence $(\chi_n)_{n\in\w}$ in $E'$ converges uniformly on $B$, that is
$
\lim_n \|\chi_n\|_B =0.
$
A Banach space $E$ is called {\em Gelfand--Phillips} if every limited set in $E$ is precompact. A subset $A$ of $E$ is {\em precompact} if its closure in $E$ is compact.

In \cite{Gelfand} Gelfand proved that every separable Banach space is Gelfand--Phillips. The condition of being separable is essential. Indeed, Phillips \cite{Phillips} showed that the space $\ell_\infty=C(\beta\w)$ is not Gelfand--Phillips, where $\beta\w$ is the Stone--\v{C}ech compactification of the discrete space $\w$ of nonnegative integers.

 Bourgain and Diestel  \cite{BourDies} defined an operator $T:L\to E$ between Banach spaces to be {\em limited} if $T(B_L)$ is a limited subset of $E$, or equivalently, if the adjoint operator $T^\ast$ is weak$^\ast$-norm sequentially continuous. In \cite{Drewnowski}, Drewnowski observed the following characterization playing a considerable  role in recognizing Gelfand--Phillips spaces.

\begin{theorem} \label{t:Drew-GP}
For a Banach space $E$ the following assertions are equivalent:
\begin{enumerate}
\item[{\rm (i)}] $E$ is Gelfand--Phillips;
\item[{\rm (ii)}] every limited weakly null sequence in $E$ is norm null;
\item[{\rm (iii)}] every limited operator with range in $E$ is compact.
\end{enumerate}
\end{theorem}
\noindent It immediately follows from (ii) that every Schur space (in particular, $\ell_1(\Gamma)$ for some infinite set $\Gamma$) is Gelfand--Phillips. 

Gelfand--Phillips spaces were intensively studied by many authors, see for example \cite{Drewnowski,DrewEm,Leung,Schlumprecht-Ph,Schlumprecht-C,SinhaArora} and more recent articles \cite{CGP,GhLe,Jos,SaMo}. Another direction for studying the  Gelfand--Phillips property is to characterize Gelphand--Phillips spaces that belong to some important classes of Banach spaces.
 Since every Banach space is (isometrically) embedded in a $C(K)$-space, it is important to recognize Gelfand--Phillips spaces among Banach spaces of continuous functions in terms of the compact space $K$. Some sufficient conditions on compact spaces $K$ to have Gelfand--Phillips space $C(K)$  were obtained by Drewnowski \cite{Drewnowski}, Drewnowski and Emmanuele \cite{DrewEm}, and by Schlumprecht in \cite{Schlumprecht-Ph,Schlumprecht-C}.

In the first main result of the paper (Theorem  \ref{t:Banach-eJNP-1}) we obtain several new characterizations of Gelfand--Phillips spaces from which we deduce some sufficient conditions of being a Gelfand--Phillips space (Corollary \ref{c:Banach-GP}). Our approach involves the family $\BNP(E)$ of all {\em bounded non-precompact} subsets of a Banach space $E$, instead of {\em limited} sets in $E$. This approach  leads us to the following characterization of the Gelfand--Phillips property: {\em A Banach space $E$ is Gelfand--Phillips if and only if for every $B\in\BNP(E)$, there is a weak$^\ast$ null sequence $(\chi_n)_{n\in\w}$ in $E'$ such that $\|\chi_n\|_B \not\to 0$}. The importance of this reformulation is that it gives not only a new characterization of the Gelfand--Phillips property, but it also allows to introduce and study a {\em strong} version of that property, see  Definition \ref{def:Banach-strong-GP} below.

Analysing the aforementioned characterization of the Gelfand--Phillips property, one can naturally ask:
When does there exist a {\em fixed} weak$^\ast$ null sequence $(\chi_n)_{n\in\w}$ in the dual $E'$ such that $\|\chi_n\|_B \not\to 0$  for {\em every} $B\in \BNP(E)$? This question motivates to introduce the following strong version of the Gelfand--Phillips property.

\begin{definition} \label{def:Banach-strong-GP}
A Banach space $E$ is defiend to have the {\em strong Gelfand--Phillips property} if $E$ admits a  weak$^\ast$ null-sequence $(\chi_n)_{n\in\w}$ in $E'$ such that $\|\chi_n\|_B \not\to 0$ for every $B\in\BNP(E)$. In this case we will say that $E$ is {\em strongly Gelfand--Phillips}.\qed
\end{definition}

Strong Gelfand--Phillips spaces are studued in Section~\ref{sec:strong-GP}. It turns out that the class of such spaces is rather narrow: according to Theorem \ref{t:Banach-strong-GP}, a  Banach space $E$ is strongly Gelfand--Phillips if and only if it embeds into $c_0$; by Theorem \ref{t:Banach-C(K)-strong-GP}, for a compact space $K$ the Banach space $C(K)$ is strongly Gelfand--Phillips if and only if $C(K)$ is isomorphic to $c_0$ if and only if $K$ is a countable compact space of finite scattered height.


\section{A characterization of Gelfand--Phillips spaces} \label{sec:e-JN}




A topological space $X$ is defined to be
\begin{itemize}
\item {\em sequentially compact} if every sequence in $X$ contains a convergent subsequence;
\item {\em selectively sequentially pseudocompact at} a subset $A\subseteq X$ if for any open sets $U_n\subseteq X$, $n\in\w$, intersecting the set $A$,  there exists a sequence $(x_n)_{n\in\w}\in\prod_{n\in\w}U_n$ containing a convergent subsequence;
\item {\em selectively sequentially pseudocompact} if $X$ is sequentially sequentially pseudocompact at $X$, see \cite{DAS1}.
\end{itemize}
It is clear that every selectively sequentially pseudocompact space $X$ is pseudocompact, but the converse is not true in general since $X$ must contain non-trivial convergent sequences. In particular, the Stone--\v{C}ech compactification $\beta D$ of an infinite discrete space $D$ is not selectively sequentially pseudocompact.
Compact selectively sequentially pseudocompact spaces form the class $\KK''$ introduced by Drewnowski and Emmanuele \cite{DrewEm}. 

A non-trivial class of  selectively sequentially pseudocompact spaces is the class of Valdivia compact spaces widely studied in Functional Analysis.
Let us recall that a compact space $K$ is {\em Valdivia compact} if $K$ is homeomorphic to a  subspace $X$ of a Tychonoff cube $[-1,1]^\kappa$ such that for the set $\Sigma :=\{x\in[-1,1]^\kappa:|\supp(x)|\leq\w\}$, the intersection $X\cap\Sigma$ is dense in $X$. Since the space $\Sigma$ is sequentially compact, every Valdivia compact space is selectively sequentially pseudocompact. It should be mentioned that Banach spaces whose dual unit ball is Valdivia compact in the weak$^\ast$ topology form an important and well-studied class of Banach spaces, see \cite{Kalenda}.


Let $E$ be a Banach space. For simplicity of notations, the space $E$ endowed with the weak topology is denoted by $E_w$, and $E'_{w^\ast}$ denotes the dual space $E'$ with the weak$^\ast$  topology.
A bounded subset $S$ of $E'$ is called {\em norming} if the formula $\|x\|_S=\sup_{\chi\in S}|\chi(x)|$ determines an equivalent norm on $E$, see \cite[p.160]{fabian-10}. We generalize this classical notion as follows. Let $B$ be a subset of $E$. A bounded subset $S$ of $E'$ is defined to be {\em $B$-norming} if there exist a positive constant $\lambda$ such that
\[
\tfrac1\lambda\|x\|_S \leq \|x\|\leq \lambda\|x\|_S \quad \mbox{ for every }\; x\in B.
\]
In some papers (e.g.  \cite{Schlumprecht-C}) $E$-norming sets are called  $E$-norming up to a constant.

In the following theorem, which is the main result of this section,  we give several new characterizations of Gelfand--Phillips spaces.


\begin{theorem} \label{t:Banach-eJNP-1}
For a Banach space $E$ the following assertions are equivalent:
\begin{enumerate}
\item[{\rm (i)}] $E$ is Gelfand--Phillips.
\item[{\rm (ii)}] For every $X\in\BNP(E)$ there exists an operator $T:E\to c_0$ such that $T(X)$ is not  precompact in the Banach space $c_0$.
\item[{\rm (iii)}] For every $X\in\BNP(E)$, there is a weak$^\ast$ null-sequence $(\chi_n)_{n\in\w}$ in $E'$ such that $\|\chi_n\|_X \not\to 0$.
\item[{\rm (iv)}] For every $B\in \BNP(E)$ there exist an infinite  subset $X \subseteq B$ such that the space $E'_{w^\ast}$ is selectively sequentially pseudocompact at some $(X-X)$-norming set $S\subseteq E'$.
\end{enumerate}
\end{theorem}

\begin{proof} (i)$\Rightarrow$(ii) Assume that $E$ is Gelfand--Phillips and take any set $X\in\BNP(E)$. Since $E$ is Gelfand--Phillips, the non-precompact set $X$ is not limited, which means that $\|\chi_n\|_X\not\to 0$ for some weak$^\ast$ null-sequence $(\chi_n)_{n\in\w}$ in $E'$. By the Banach--Steinhaus Uniform Boundedness Principle, the weak$^\ast$ bounded set $\{\chi_n\}_{n\in\w}$ is norm-bounded in $E'$, which implies that the linear map
\[
T:E\to c_0,\quad T:x\mapsto\big(\chi_n(x)\big)_{n\in\w},
\]
is bounded. It remains to show that $T(X)$ is not precompact in $c_0$. Assuming that $T(X)$ is precompact, for every $\e>0$ we can find a finite set $F\subseteq X$ such that $T(F)$ is an $\e$-net in $T(X)$ (which means that for every $x\in T(X)$ there exists $y\in T(F)$ such that $\|x-y\|<\e$). Since $(\chi_n)_{n\in\w}$ is weak$^\ast$ null
and $F$ is finite, there exists $m\in\w$ such that $\sup_{n\ge m}\max_{y\in F}|\chi_n(y)|<\e$. For every $x\in X$ choose $y\in F$ with $\|T(x)-T(y)\|<\e$ and conclude that for every $n\ge m$ we have
\[
|\chi_n(x)|\le|\chi_n(y)|+|\chi_n(y)-\chi_n(x)|<\e+\|T(y)-T(x)\|<2\e,
\]
which implies that $\|\chi_n\|_X\to 0$ and contradicts the choice of the sequence $(\chi_n)_{n\in\w}$.
\smallskip

(ii)$\Rightarrow$(iii) Given any set $X\in\BNP(E)$, apply (ii) and find an operator $T:E\to c_0$ such that $T(X)$ is not precompact in $c_0$. For every $n\in\w$, let $\chi_n=e_n'\circ T$, where $e'_n$ is the $n$th coordinate functional in $c_0$. It is clear that the sequence $(\chi_n)_{n\in\w}$ is weak$^\ast$ null in $E'$. It remains to prove that $\|\chi_n\|_X\not\to 0$. To derive a contradiction, assume that $\|\chi_n\|_X\to 0$. Then for every $\e>0$ there exists $m\in\w$ such that $\sup_{n\ge m}\|\chi_n\|_X<\e$. Since $X$ and $T$ are bounded, the number $b=\sup\{\|T(x)\|:x\in X\}$ is finite. Since the bounded finite-dimensional set
\[
K=\big\{y\in c_0:\|y\|\leq b,\;\mbox{ and } \; e'_n(y)=0 \mbox{ for all }  n\geq m\}
\]
is compact in $c_0$, there exists a finite set $F\subseteq K$ such that for every $x\in K$ there exists $y\in F$ such that $\|x-y\|<\e$. Given any $x\in X$, let $y\in K$ be a unique element such that $e'_n(y)=\chi_n(x)$ for  all $n<m$. The choice of the number $m$ ensures that $\|y-T(x)\|=\sup_{n\ge m}|\chi_n(x)|<\e$. Since $y\in K$, there exists $z\in F$ such that $\|y-z\|<\e$. Then $\|x-z\|<2\e$ which means that $F$ is a $2\e$-net for $T(X)$ and hence the set $T(X)$ is precompact in $c_0$, which contradicts the choice of the operator $T$.
\smallskip

(iii)$\Rightarrow$(iv) Given any set $B\in\BNP(E)$,
we apply (iii) to find a weak$^\ast$ null sequence $(\chi_n)_{n\in\w}$ in $E'$ such that $\inf_{n\in\w}\|\chi_n\|_B>0$. Multiplying the elements of the sequence $(\chi_n)_{n\in\w}$ by a suitable constant, we can assume that $\inf_{n\in\w}\|\chi_n\|_B>2b$ where $b:=\sup\{\|x-y\|:x,y\in B\}>0$.  Taking into account that $(\chi_n)_{n\in\w}$ is weak$^\ast$ null, construct inductively a sequence of points $(x_k)_{k\in\w}$ in $B$ and an increasing sequence of numbers $(n_k)_{k\in\w}$ such that for every $k\in\w$ the following two conditions are satisfied:
\begin{itemize}
\item $|\chi_{n_k}(x_k)|>2b$;
\item $|\chi_{n_k}(x_i)|<b$ for every $i<k$.
\end{itemize}
The triangle inequality ensures that $|\chi_{n_k}(x_k-x_i)|>b$ for every $i<k$.

Let $X=\{x_k:k\in\w\}$. Since the sequence $(\chi_{n_k})_{k\in\w}$ is weak$^*$ null, the space $E'_{w^\ast}$ is  selectively sequentially pseudocompact at the set $S=\{\chi_{n_k}:k\in\w\}$. We claim that $S$ is $(X-X)$-norming. By the Banach--Steinhaus Uniform Boundedness Principle, the weak$^\ast$ null sequence $(\chi_n)_{n\in\w}$ is norm bounded in $E'$ and so is the set $S\subseteq E'$. Then the positive constant $C=\sup\{\|\chi\|:\chi\in S\}$ is well-defined and, for every $z\in E$, we have $\|z\|_S=\sup_{\chi\in S}|\chi(z)|\leq C\cdot\|z\|$. If $z\in X-X$, then $z=x_k-x_i$ for some $k,i\in\w$. If $k=i$, then $z=0$ and $\|z\|_S=0=\|x\|$. If $k\ne i$, then
\[
\|z\|_S\ge\max\{|\chi_{n_i}(x_k-x_i)|,\;|\chi_{n_k}(x_k-x_i)|\}>b\ge \|x_k-x_i\|=\|z\|.
\]
Therefore,
\[
\|z\|\le\|z\|_S\le C\|z\|
\]
for every $z\in X-X$, which means that the set $S$ is $(X-X)$-norming.
\smallskip

(iv)$\Rightarrow$(i) Assuming that $E$ is not Gelfand--Phillips, we can find a limited set $B\in\BNP(E)$. Since $B$ is not precompact, there exists $\e>0$ such that for every finite subset $F\subseteq B$ there exists $x\in B$ such that $\|x-y\|>\e$ for all $y\in F$. Using this property of $\e$, we can inductively construct a sequence $\{z_n\}_{n\in\w}\subseteq B$ such that $\|z_n-z_i\|>\e$ for every $i<n$. It is clear that the set $\{z_n:n\in\w\}$ is not precompact. By (iv),  there exist an infinite  subset $X \subseteq \{z_n:n\in\w\}$ such that the space $E'_{w^\ast}$ is selectively sequentially pseudocompact at some $(X-X)$-norming set $S\subseteq E'_{w^\ast}$. Then there exists a positive real constant $c$ such that $\|z\|_S\ge c\,\|z\|$ for every $z\in X-X$. Consequently, for any distinct elements $x,y\in X$, we have
\[
\|x-y\|_S=\sup_{f\in S}|f(x-y)|\ge c\,\|x-y\|>c\e.
\]
Write the set $X$ as $\{x_n\}_{n\in\w}$ for pairwise distinct points $x_n$, and for every $n<m$, choose a linear functional $f_{n,m}\in S$ such that $|f_{n,m}(x_m-x_n)|>c\e$.

The selective sequential pseudocompactness of $E'_{w^\ast}$ at $S$ implies that the set $S$ in bounded in $E'_{w^\ast}$ and hence bounded in $E$ (by the Banach--Steinhaus Uniform Boundedness Principle). Therefore the set $\{f(x):f\in S,\;x\in X\}$ has compact closure $K$ in the field $\mathbb F$ (of real or complex numbers).

Using the sequential compactness of the compact metrizable space $K^X$, we can construct a decreasing sequence $\{\Omega_n\}_{n\in\w}$ of infinite sets in $\w$ such that for every $n\in\w$, the sequence $\{f_{n,m}{\restriction}_X\}_{m\in\Omega_n}$ converges to some function $f_n\in K^X$.

Choose an infinite set $\Omega\subseteq\w$ such that $\Omega\setminus\Omega_n$ is finite for every $n\in\w$. Since the compact metrizable space $K^X$ is sequentially compact, we can replace $\Omega$ by a smaller infinite set and additionally assume that the sequence $\{f_n\}_{n\in\Omega}$ converges to some element $f_\infty\in K^X$. Since the set $\{f_\infty(x_n)\}_{n\in\Omega}\subseteq K$ admits a finite cover by sets of diameter $<\tfrac{1}{4}c\e$, we can replace $\Omega$ by a suitable infinite subset and additionally assume that the set $\{f_\infty(x_n)\}_{n\in\Omega}$ has diameter $<\tfrac{1}{4}c\e$.

It follows that  the function $f_\infty\in K^X$ belongs to the closure of the set $\{f_{n,m}{\restriction}_X:n,m\in \Omega,\; n<m\}$.
Since the space $K^X$ is first-countable, we can choose a sequence
\[
\{(n_i,m_i)\}_{i\in\w}\subseteq\{(n,m)\in\Omega\times\Omega:n<m\}
\]
such that the sequence $\{f_{n_i,m_i}{\restriction}_X\}_{i\in\w}$ converges to $f_\infty$. Since $\IF^X$ is metrizable, for every $i\in\w$ the element $f_{n_i,m_i}{\restriction}_X$ of $K^X\subseteq\IF^X$ has an open neighborhood $V_i\subseteq \IF^X$  such that the sequence $\{V_i\}_{i\in\w}$ converges to $f_\infty$ in the sense that every neighborhood of $f_\infty$ in $\IF^X$ contains all but finitely many sets $V_i$. Since $|f_{n_i,m_i}(x_{n_i}-x_{m_i})|>c\e$, we can replace each set $V_i$ by a smaller neighborhood of $f_{n_i,m_i}{\restriction}_X$ and additionally assume that $|g(x_{n_i})-g(x_{m_i})|>c\e$ for every $g\in V_i$. For every $i\in\w$, consider the open neighborhood $W_i:=\{f\in E'_{w^\ast}:f{\restriction}_X\in V_i\}$ of the functional $f_{n_i,m_i}$ in the space $E'_{w^\ast}$.

Since $E'_{w^\ast}$ is selectively sequentially pseudocompact  at $S$, there exists a convergent sequence $(g_k)_{k\in\w}\subseteq E'_{w^\ast}$ and an increasing number sequence $\{i(k)\}_{k\in\w}$ such that $g_k\in W_{i(k)}$ for every $k\in\w$. Let $g_\infty\in E'_{w^\ast}$ be the limit of the sequence $(g_k)_{k\in\w}$. The continuity of the restriction operator $E'_{w^\ast}\to\IF^X$, $f\mapsto f{\restriction}_X$, and the choice of the open sets $V_i$, $i\in\w$, guarantee that $g_\infty{\restriction}_X=f_\infty$. Consequently, the sequence $(g_k{\restriction}_X)_{k\in\w}$ converges to $g_\infty{\restriction}_X =f_\infty$ in $\IF^X$.

Then for every $k\in \w$, we can find a number $j_k>k$ such that
\[
\max\{|f_\infty(x_{n_{i(k)}})-g_{j_k}(x_{n_{i(k)}})|, |f_\infty(x_{m_{i(k)}})-g_{j_k}(x_{m_{i(k)}})|\}<\tfrac{1}{8}c\e.
\]

For every $k\in\w$, consider the functional $\mu_k:=g_k-g_{j_k}\in E'$ and observe that the sequence $\{\mu_k\}_{k\in\w}$ converges to zero in $E'_{w^\ast}$. On the other hand,  for every $k\in \w$, the choice of the sequence $(j_k)_{k\in\w}$,
the inequality $\diam \{f_\infty(x_n)\}_{n\in\Omega}<\frac14c\e$, and the inclusion $g_k{\restriction}_X\in V_{i(k)}$  imply
\[
\begin{aligned}
|g_{j_k}(x_{n_{i(k)}}) & -g_{j_k}(x_{m_{i(k)}})|\\
& =|g_{j_k}(x_{n_{i(k)}})-f_\infty(x_{n_{i(k)}})+f_\infty(x_{m_{i(k)}})-g_{j_k}(x_{m_{i(k)}})+f_\infty(x_{n_{i(k)}})-f_\infty(x_{m_{i(k)}})|\\
& \leq|g_{j_k}(x_{n_{i(k)}})-f_\infty(x_{n_{i(k)}})|+|f_\infty(x_{m_{i(k)}})-g_{j_k}(x_{m_{i(k)}})|+|f_\infty(x_{n_{i(k)}})-f_\infty(x_{m_{i(k)}})|\\
&\leq \tfrac{1}{8}c\e+\tfrac{1}{8}c\e+\tfrac{1}{4}c\e=\tfrac{1}{2}c\e
\end{aligned}
\]
and
\[
\begin{aligned}
|\mu_k(x_{n_{i(k)}}) -\mu_k (x_{m_{i(k)}})|& =|g_{k}(x_{n_{i(k)}})-g_{j_k}(x_{n_{i(k)}})-g_k(x_{m_{i(k)}})+g_{j_k}(x_{m_{i(k)}})|\\
& \geq |g_{k}(x_{n_{i(k)}})-g_k(x_{m_{i(k)}})|-|g_{j_k}(x_{n_{i(k)}})-g_{j_k}(x_{m_{i(k)}})|>c\e-\tfrac{1}{2}c\e=\tfrac{1}{2}c\e.
\end{aligned}
\]
Then, for every $k\in \w$,
\[
\sup_{x\in B}|\mu_k(x)|\geq \max\{|\mu_k(x_{n_{i(k)}})|,|\mu_k(x_{m_{i(k)}})|\}\geq\tfrac{1}{4}c\e,
\]
witnessing that the set $B$ is not limited. This contradiction shows that the Banach space $E$ is Gelfand--Phillips.
\end{proof}

In Corollary \ref{c:Banach-GP} below we apply Theorem \ref{t:Banach-eJNP-1} to obtain some new and several well-known sufficient conditions on Banach spaces to be a Gelfand--Phillips space. For this we should recall some additional definitions.

Following Castillo, Gonz\'{a}lez and Papini \cite{CGP}, we define a Banach space $E$ to be {\em separably weak$^\ast$-extensible} if  any weak$^\ast$ null sequence in the dual of any separable subspace of $E$ admits a subsequence which can be extended to a weak$^\ast$ null sequence in $E'$. This class coincides with the following class of Banach spaces introduced by Correa and Tausk \cite{CT13,CT14}: A Banach space $E$ has the {\em separable $c_0$-exension property} if every operator $T:X\to c_0$ defined on a separable Banach subspace $X\subseteq E$ can be extended to an operator $\bar T:E\to c_0$. By \cite{CT13,CT14}, the class of Banach spaces with the separable $c_0$-extension property includes all weakly compactly generated Banach spaces and  all Banach spaces with the separable complementation property (= every separable subspace is contained in a separable complemented subspace).

For a compact space $K$, let
\[
P(K):=\{\mu\in C(K)':\|\mu\|=\mu(\mathbf{1}_K)=1\}
\]
be the space of probability measures on $K$, endowed with the weak$^\ast$ topology, inherited from $C(K)'_{w^\ast}$. It is well-known that the space $P(K)$ is compact. Identifying each $x\in K$ with the Dirac measure $\delta_x:C(K)\to\IF$, $\delta_x:f\mapsto f(x)$, we identify $K$ with a closed subspace of $P(K)\subseteq C(K)'_{w^\ast}$. Observe that the set  $\{\delta_x:x\in K\}$ is $C(K)$-norming.


\begin{corollary} \label{c:Banach-GP}
A Banach space $E$ is Gelfand--Phillips if  one of the following conditions holds:
\begin{enumerate}
\item[{\rm (i)}] the closed unit ball  $B_{E'}$ endowed with the weak$^\ast$ topology is selectively sequentially pseudocompact;
\item[{\rm (ii)}]  {\rm(\cite{Gelfand})} $E$ is separable;
\item[{\rm (iii)}] {\rm(\cite[Prop.~2]{CGP})} $E$  has the separable $c_0$-extension property \textup{(}$\Leftrightarrow E$ is  separably weak$^\ast$-extensible$)$;
\item[{\rm (iv)}]  {\rm(cf. \cite[Th.~2.2]{Drewnowski} and \cite[Prop.~2]{Schlumprecht-C})} the space $E'_{w^\ast}$ is selectively sequentially pseudocompact at some $E$-norming set $S\subseteq E'$;
\item[{\rm(v)}]  {\rm(\cite[Th.~4.1]{DrewEm})} $E=C(K)$ for some  compact selectively sequentially pseudocompact space $K$;
\item[{\rm (vi)}] $E=C(K)$ for some compact space $K$ such that $P(K)$ is  selectively sequentially pseudocompact at some set $A\subseteq P(K)$ containing $K$.
\end{enumerate}
\end{corollary}

\begin{proof}
(i) If $B_{E'}$ is selectively sequentially pseudocompact in the weak$^\ast$ topology, then by the equivalence  (i)$\Leftrightarrow$(iv) in Theorem~\ref{t:Banach-eJNP-1}, $E$ is Gelfand--Phillips because $B_{E'}$ is $E$-norming.
\smallskip

(ii) If $E$ is separable, then $B_{E'}$ is compact metrizable and hence selectively sequentially pseudocompact in the weak$^\ast$ topology. By (i), $E$ is Gelfand--Phillips.
\smallskip

(iii) Assume that $E$ has the separable $c_0$-extension property. In order to apply Theorem~\ref{t:Banach-eJNP-1}, it suffices for every $B\in \BNP(E)$ to find an operator $T:E\to c_0$ such that $T(B)$ is not precompact in $c_0$. Given any $B\in\BNP(E)$, find a separable Banach subspace $X\subseteq E$ such that the set $X\cap B$ is not precompact in $X$. By (ii), the separable Banach space $X$ is Gelfand--Phillips. By Theorem~\ref{t:Banach-eJNP-1}, there exists an operation $T:X\to c_0$ such that the set $T(X\cap B)$ is not precompact in $c_0$. Since $E$ has the separable $c_0$-extension property, the operator $T$ can be extended to an operator $\bar T:E\to c_0$. It is clear that the set $\bar T(B)\supseteq \bar T(X\cap B)$ is not precompact in $c_0$.
\smallskip

(iv) Assume that the space $E'_{w^\ast}$ is selectively sequentially pseudocompact at some $E$-norming set $S\subseteq E'$. By (iv) of Theorem~\ref{t:Banach-eJNP-1}, the Banach space $E$ is Gelfand--Phillips.
\smallskip

(v) Assume that $E=C(K)$ for some compact selectively sequentially pseudocompact space $K$. Identify $K$ with the set of Dirac measures in $E'_{w^\ast}$. Then $K$ is an $E$-norming set and $E'_{w^\ast}$ is  selectively sequentially pseudocompact at $K$. By  (iv) of Theorem~\ref{t:Banach-eJNP-1}, the space $E$ is Gelfand--Phillips.
\smallskip

(vi) Assume that $E=C(K)$ for some compact space $K$ such that $P(K)$ is  selectively sequentially pseudocompact at some set $A\subseteq P(K)$ containing $K$. Identify $P(K)$ with the subspace of positive  functionals in $E'_{w^\ast}$ and observe that $E'_{w^\ast}$ is selectively sequentially pseudocompact at $A$ and $A\supseteq K$ is $E$-norming. By Theorem~\ref{t:Banach-eJNP-1} the space $E$ is Gelfand--Phillips.
\end{proof}

Sinha and Arora \cite[Corollary~2.4]{SinhaArora} showed that for any Valdivia compact space, the Banach spave $C(K)$ is Gelfand--Phillips. As we mentioned above, every Valdivia compact is selectively sequentially pseudocompact, and hence their result follows from  Corollary \ref{c:Banach-GP}(v). 

It was noticed in \cite{CGP} that the condition to have the separable $c_0$-extension property in (iii) of Corollary \ref{c:Banach-GP} is only sufficient to have the Gelfand--Phillips property. Below we give a concrete example. Let us recall that the {\em split interval\/} $\ddot\II$ is the space $[0,1]\times\{0,1\}$ endowed with the interval topology generated by the lexicographic order $\le$ defined by $(x,i)\le (y,j)$ if and only if either $x<y$ or $x=y$ and $i\le j$. It is well known that the split interval is first-countable and separable but not metrizable.

\begin{example}\label{exa:split-eJNP}
For the split interval $\ddot\II$, the Banach space $C(\ddot\II)$ is Gelfand--Phillips but fails to have the separable $c_0$-extension property.
\end{example}

\begin{proof}
Being compact and first-countable, the split interval $\ddot\II$ is sequentially compact and hence selectively sequentially pseudocompact. By Corollary \ref{c:Banach-GP}(vi), the Banach space $C(\ddot\II)$ is Gelfand--Phillips.  Since the space $\ddot\II$ is linearly ordered, separable and non-metrizable, we can apply Theorem 2.2 of \cite{CT14} and conclude that the Banach space $C(\ddot\II)$ does not have the separable $c_0$-extension property.
\end{proof}

By the Phillips result \cite{Phillips}, the Banach space $C(\beta\w)$ is not Gelfand--Phillips. Below we generalize this result.
Recall that a Tychonoff space $X$ is  an {\em $F$-space} if every functionally open set $A$ in $X$ is {\em $C^\ast$-embedded} in the sense that every bounded continuous function $f:A\to \IR$ has a continuous extension $\bar f:X\to\IR$. For  numerous equivalent conditions for a Tychonoff space $X$ to be an $F$-space, see \cite[14.25]{GiJ}. In particular, the Stone--\v{C}ech compactification $\beta \Gamma$ of any discrete space $\Gamma$ is a compact $F$-space.

\begin{example}\label{exa:F-space-eJNP}
For any infinite compact $F$-space $K$, the Banach space $C(K)$ is not Gelfand--Phillips. 
\end{example}

\begin{proof} 
Being infinite, the Tychonoff space $K$ contains a sequence $\{V_n\}_{n\in\w}$ of nonempty pairwise disjoint open sets. For every $n\in\w$, fix a point $v_n\in V_n$ and a continuous function $f_n:K\to [0,1]$ such that $f_n(v_n)=1$ and $f_n(K\SM V_n)=\{0\}$. Consider the operator $T:c_0\to C(K)$ assigning to each sequence $x=(x_n)_{n\in\w}\in c_0$ the continuous function $T(x)=\sum_{n\in\w}x_n\cdot f_n$, and observe that $T$ is an isometric embedding of $c_0$ into $C(K)$. By Corollary 4.5.9 of \cite{Dales-Lau}, the Banach space $C(K)$ has the Grothendieck property, which means that the identity map $C(K)'_{w^\ast}\to C(K)'_w$ is sequentially continuous (where $C(K)'_w$ denotes the dual space of $C(K)$  endowed with the weak topology).

Since the operator $T:c_0\to C(K)$ is an embedding, the image $B:=T(B_{c_0})$ is bounded and not precompact in $C(K)$, i.e., $B\in\BNP\big(C(K)\big)$. Assuming that $C(K)$ is Gelfand--Phillips, we can find a weak$^\ast$ null sequence $S=\{\mu_n\}$ in $E'$ such that $\|\mu_n\|_{B}\not\to 0$. Since the identity map  $C(K)'_{w^\ast}\to C(K)'_w$ is sequentially continuous, we obtain that the sequence $S$ converges to zero in the weak topology of the dual Banach space $C(K)'$. Then, for the adjoint operator $T^\ast:C(K)'_w\to (c'_0)_w=(\ell_1)_w$, the sequence $\{T^\ast(\mu_n)\}_{n\in\w}$ converges to zero in the weak topology of the Banach space $\ell_1$. By the Schur Theorem \cite[VII]{Diestel}, this sequence converges to zero in norm. For every $n\in\w$ and $x\in B_{c_0}$, we have
\[
\|\mu_n\|_{B}=\sup_{x\in B_{c_0}}\big|\mu_n\big(T(x)\big)\big|=\sup_{x\in B_{c_0}}|T^\ast(\mu_n)(x)|=\|T^\ast(\mu_n)\|\to 0,
\]
which contradicts the choice of the sequence $(\mu_n)_{n\in\w}$. Thus the Banach space $C(K)$ is not Gelfand--Phillips.
\end{proof}

It is known that Gelfand--Phillips spaces are not preserved by taking quotients, see \cite{Schlumprecht-Ph}. Below we present a simple example witnessing this fact.

\begin{example}\label{exa:XY-eJNP}
There are compact Hausdorff spaces $X\subseteq Y$ such that the Banach space $C(Y)$ is Gelfand--Phillips but the Banach space $C(X)$ is not Gelfand--Phillips. In particular, a quotient of a Gelfand--Phillips Banach space can fail to be Gelfand--Phillips.
\end{example}

\begin{proof}
In the Cantor cube $Y:=\{0,1\}^{\mathfrak{c}}$ of weight $\mathfrak{c}$ take a subspace $X$, homeomorphic to $\beta\w$. Being Valdivia compact, the Cantor cube $\{0,1\}^{\mathfrak{c}}$ is selectively sequentially pseudocompact and, by  Corollary \ref{c:Banach-GP}(v), the Banach $C(Y)$ is Gelfand--Phillips. By Example~\ref{exa:F-space-eJNP}, the Banach space $C(X)$ is not Gelfand--Phillips.
\end{proof}

It is worth mentioning that, by results of Schlumprecht \cite{Schlumprecht-Ph,Schlumprecht-C},  the Gelfand--Phillips property is not a three space property (see also Theorem 6.8.h in \cite{CG}).
\smallskip

We finish this section with  two questions. The first one is related to a known open problem of characterizing Banach spaces $E$ for which the dual unit ball $B_{E'}$ is weak$^\ast$ sequentially compact  (for historical remarks and the latest results,  see \cite{MaCe}).

\begin{problem} \label{prob:Banach-near}
Characterize Banach spaces $E$ for which the dual unit ball $B_{E'}$ is weak$^\ast$ selectively sequentially pseudocompact.
\end{problem}


The following problem is motivated by the conditions (v) and (vi) of  Corollary \ref{c:Banach-GP}.

\begin{problem} \label{prob:Banach-Efimov}
Is there an infinite compact space $K$ whose space of probability measures $P(K)$ is selectively sequentially pseudocompact but $K$ contains no non-trivial convergent sequences?
\end{problem}

\begin{remark} By \cite{BG}, under Jensen's Diamond Principle $\diamondsuit$ (which is stronger than the Continuum Hypothesis),  there exists an infinite compact space $K$ such that the compact space $P(K)$  is selectively sequentially pseudocompact  but $K$ contains no topological copies of the spaces $\beta\w$ and $\w+1=\w\cup\{\w\}$. By Corollary \ref{c:Banach-GP}(vi), the Banach space $C(K)$  is Gelfand--Phillips, yet $K$ contains no nontrivial convergent sequences (a CH-example of a compact space with these two properties has been constructed by Schlumprecht in \cite[\S~5.4]{Schlumprecht-Ph}). The space $K$ shows that Problem~\ref{prob:Banach-Efimov} has an affirmative answer under $\diamondsuit$. So, this problem essentially asks about the existence of a ZFC-example. It should be mentioned that infinite compact spaces containing no topological copies of the spaces $\beta\w$ and $\w+1$ are called {\em Efimov}. The problem of the existence of Efimov compact spaces in ZFC is one of major unsolved problems of Set-Theoretic Topology, see \cite{Nyikos}, \cite{Hart}.\qed
\end{remark}


\section{Banach spaces with the strong Gelfand--Phillips property} \label{sec:strong-GP}



Below we characterize Banach spaces with the strong Gelfand--Phillips property.

\begin{theorem} \label{t:Banach-strong-GP}
A Banach space $E$ is strongly Gelfand--Phillips if and only if it embeds into $c_0$.
\end{theorem}

\begin{proof} Assuming that $E$ is strongly Gelfand--Phillips, find a weak$^\ast$ null sequence $(\chi_n)_{n\in\w}$ in $E'$ such that $\|\chi_n\|_B\not\to 0$ for any $B\in\BNP(E)$. Therefore every bounded subset of the Banach subspace $Z=\bigcap_{n\in\w}\chi_n^{-1}(0)$ of $E$ is precompact, which implies that the subspace $Z$ is finite-dimensional. Unifying the weak$^\ast$ null sequence $\{\chi_n\}_{n\in\w}$ with a finite set of functionals separating points of the finite-dimensional space $Z$, we can assume that $Z=\{0\}$. In this case the
 linear map
\[
T:E\to c_0, \quad T:x\mapsto\big(\chi_n(x)\big)_{n\in\w},
\]
is injective. By  the Banach--Steinhaus Uniform Boundedness Principle, the map $T$ is continuous. Assuming that the operator $T$ is not a topological embedding, we can find a sequence $\{x_n\}_{n\in\w}\subseteq E$ of elements of norm $1$ such that $T(x_n)\to 0$.

We claim that the bounded set $B=\{x_n\}_{n\in\w}$ is not precompact in $E$. Indeed, in the opposite case, by the completeness of $E$, the sequence $\{x_n\}_{n\in\w}$ would contain a subsequence $\{x_{n_k}\}_{k\in\w}$ that converges in $E$ to some element $x_\infty\in E$ of norm $\|x_\infty\|=1$. The continuity of the operator $T$ ensures that $T(x_\infty)=\lim_{n\to\infty}T(x_n)=0$, which contradicts the injectvity of $T$. This contradiction shows that the set  $B$ is not precompact in $E$.

Now the choice of the sequence $(\chi_n)_{n\in\w}$ ensures that the sequence $(\|\chi_n\|_B)_{n\in\w}$ does not converge to zero. Since $\lim_{n\to\infty}\|T(x_n)\|=0$,  for every $\e>0$, we can find an $n\in\w$ such that $\|T(x_i)\|=\sup_{k\in\w}|\chi_k(x_i)|<\e$ for all $i\ge n$. Since the sequence $(\chi_{n})_{n\in\w}$ weak$^\ast$ null, there exists a natural number $m\ge n$ such $|\chi_k(x_i)|<\e$ for all $i\leq n$ and $k\geq m$. Then for every $k\geq m$, we have
\[
\|\chi_k\|_B=\sup_{i\in\w}|\chi_k(x_i)|=\max\big\{\max_{i\leq n}|\chi_k(x_i)|,\sup_{i>n}|\chi_k(x_i)|\big\}<\e,
\]
which means that $\|\chi_k\|_B\to 0$. This contradiction shows that the operator $T:E\to c_0$ is a topological embedding.
\smallskip

Conversely, assume now that $E$ is a subspace of the Banach space $c_0$. 
For every $n\in\w$, let $\chi_n=e_n'{\restriction}_E$ be the restriction of the coordinate functional $e_n'\in c_0' =\ell_1$ to the subspace $E\subseteq c_0$. Clearly, $\{\chi_n\}_{n\in\w}$ is weak$^\ast$-null in $E'$. Repeating the argument of the proof of the implication (ii)$\Rightarrow$(iii) in Theorem~\ref{t:Banach-eJNP-1}, we can show that every bounded set $B\subseteq E\subseteq c_0$ with $\|\chi_n\|_B\to 0$ is precompact. Thus the sequence $(\chi_n)_{n\in\w}$ witnesses that $E$ has the  strong Gelfand--Phillips property.
\end{proof}

\begin{remark}
It is well known \textup{(}see, e.g. \cite[2.d.6]{LT}\textup{)} that the Banach space $c_0$ contains closed infinite-dimensional subspaces which are not isomorphic to $c_0$.\qed
\end{remark}

As a corollary we obtain the following three space property for the class of strongly Gelfand--Phillips spaces.

\begin{corollary}\label{c:Banach-uJNP-quotient}
Let $E$ be a Banach space and $H\subseteq E$ be a closed linear subspace. Then the Banach space $E$ is strongly Gelfand--Phillips if and only if the Banach spaces $H$ and $E/H$ are strongly Gelfand--Phillips.
\end{corollary}

\begin{proof}
If $E$ is strongly Gelfand--Phillips, then by Theorem~\ref{t:Banach-strong-GP}, $E$ can be identified with a subspace of $c_0$. By Theorem  \ref{t:Banach-strong-GP}, the Banach space $H\subseteq E\subseteq c_0$ is strongly Gelfand--Phillips. By a result of Johnson and Zippin \cite{JZ}, the quotient space $c_0/H$ is isomorphic to a subspace of $c_0$, and so is the quotient space $E/H\subseteq c_0/H$. By Theorem~\ref{t:Banach-strong-GP}, the quotient space $E/H$ is strongly Gelfand--Phillips.
\vskip3pt

Now assume that the Banach spaces $H$ and $E/H$ are strongly Gelfand--Phillips.
By Theorem~\ref{t:Banach-strong-GP}, these  spaces are isomorphic to subspaces of $c_0$. Consequently, there are isomorphic embeddings $f_1:H\to c_0$ and $f_2:E/H\to c_0$. Being isomorphic to subspaces of $c_0$, the Banach spaces $H$ and $E/H$ are separable and so is the Banach space $E$.  By (an implication of) the Sobczyk Theorem \cite[p.72]{Diestel}, the linear embedding $f_1:H\to c_0$ extends to an operator $\bar f_1:E\to c_0$. Let $q:E\to E/H$ be the quotient operator. It can be shown that  the operator $f:E\to c_0\times c_0$, $f:x\mapsto (\bar f_1(x),f_2\circ q(x))$, is an isomorphic embedding of $E$ into the Banach space $c_0\times c_0$. By Theorem~\ref{t:Banach-strong-GP}, the Banach space $E$ has is strongly Gelfand--Phillips.
\end{proof}

Let us recall that a Tychonoff space $K$ is {\em pseudocompact} if each real-valued continuous function on $K$ is bounded. Observe that for a pseudocompact space $K$, the space $C(K)$ of $\IF$-valued continuous functions on $K$ is Banach with respect to the norm $\|f\|:=\sup_{x\in K}|f(x)|$.

A topological space $X$ is {\em scattered} if each nonempty subspace of $X$ has an isolated point. For a topological space $X$, let $X^{(0)}:=X$ and let $X^{(1)}$ be the space of non-isolated points of $X$. For a non-zero ordinal $\alpha$, let $X^{(\alpha)}:=\bigcap_{\beta<\alpha}(X^{(\beta)})^{(1)}$. It is well known that a topological space $X$ is scattered if and only if $X^{(\alpha)}=\emptyset$ for some ordinal $\alpha$. The smallest ordinal $\alpha$ with $X^{(\alpha)}=\emptyset$ is called the {\em scattered height} of $X$.

By a classical result of Bessaga and Pe\l czy\'nski \cite{BP60} (see also \cite[2.14]{Ros}), {\em for a compact  space $K$, the Banach space $C(K)$ is isomorphic to $c_0$ if and only if $K$ is countable and has finite scattered height}. Using this result we prove the second main result of this section.

\begin{theorem}\label{t:Banach-C(K)-strong-GP}
For an infinite pseudocompact space $K$ the following conditions are equivalent:
\begin{enumerate}
\item[{\rm(i)}] the Banach space $C(K)$ is strongly Gelfand--Phillips;
\item[{\rm(ii)}] the Banach space $C(K)$ is isomorphic to a subspace of $c_0$;
\item[{\rm(iii)}] $K$ is compact and  countable, and the Banach space $C(K)$ is isomorphic to $c_0$;
\item[{\rm(iv)}] the space $K$ is compact, countable and has finite scattered height.
\end{enumerate}
\end{theorem}

\begin{proof}
The equivalence (i)$\Leftrightarrow$(ii) follows from Theorem~\ref{t:Banach-strong-GP}, and the implication  (iii)$\Rightarrow$(ii) is trivial.
\smallskip

(ii)$\Rightarrow$(iii) Assume that the Banach space $C(K)$ is isomorphic to a subspace of $c_0$. Then the Banach space $C(K)$ has separable dual Banach space $C(K)'$. Identifying each point $x\in K$ with the Dirac measure supported at $x$, we see that $\|x-y\|=2$ for any distinct points $x,y\in K\subseteq C(K)'$. Now the separability of the dual Banach space $C(K)'$ implies that the Tychonoff space $K$ is countable. Therefore $K$ is Lindel\"of, and since $K$ is also pseudocompact, it  is compact by Theorems 3.10.21 and 3.10.1 of \cite{Eng}. By the Bessaga--Pe\l czy\'nski Theorem \cite[2.14]{Ros}, the Banach space $C(K)$ is isomorphic to $C[0,\w^{\w^\alpha}]$ for some ordinal $\alpha\geq 0$. Theorem 2.15 of \cite{Ros} implies that the Banach space $C[0,\w^{\w^\alpha}]$ has Szlenk index $\Sz(C[0,\w^{\w^\alpha}])=\w^{\alpha+1}$ and, by Proposition 2.27 in \cite{Ros}, $\Sz(c_0)=\w$. By Corollary 2.19  in \cite{Ros}, the Banach space $C(K)$, being isomorphic to a subspace of $c_0$, has Szlenk index $\Sz(C(K))\leq \Sz(c_0)$.  Then
\[
\w^{\alpha+1}=\Sz(C[0,\w^{\w^\alpha}])=\Sz(C(K))\le \Sz(c_0)=\w
\]
implies that $\alpha=0$. Therefore $C(K)$ is isomorphic to $C[0,\w^{\w^0}]=C[0,\w]$, which is isomorphic to $c_0$.
\smallskip

(iii)$\Rightarrow$(iv) Since  $C(K)$ is isomorphic to $c_0$, the compact countable space $K$ has finite scattered height by Theorem 2 in \cite{BP60} (see also \cite[2.14]{Ros}).
\smallskip

(iv)$\Rightarrow$(iii) If $K$ is compact, countable and has finite scattered height, then the Banach space $C(K)$ is isomorphic to $c_0$ by the Bessaga--Pe\l czy\'nski theorem \cite{BP60}.
\end{proof}

\section{Acknowledgements}

The authors express their sincere thanks to
Tomek Kania for the suggestion 
 to apply the Szlenk index in the proof of Theorem~\ref{t:Banach-C(K)-strong-GP} and pointing out 
 the papers \cite{CT13,CT14} devoted to the separable $c_0$-extension property.  


\end{document}